\documentclass[12pt,twoside,a4paper]{amsart}
\usepackage{amssymb}
\usepackage{graphicx}
\usepackage{amsmath}
\usepackage{amsmath}
\usepackage{amsthm,amssymb,latexsym}
\usepackage{t1enc}
\usepackage{mathrsfs}

\usepackage{graphicx}
\usepackage{xypic}

\date{\today}

\usepackage[latin1]{inputenc}
\usepackage[T1]{fontenc}

\def\End{{\rm End}}

\def\deg{\text{deg}\,}

\def\dbar{\bar\partial}

\def\R{{\mathbb R}}
\def\C{{\mathbb C}}
\def\P{{\mathbb P}}

\def\Im{{\rm Im\, }}

\def\O{{\mathcal O}}

\def\U{{\mathcal U}}

\def\codim{\text{codim}\,}

\def\supp{\text{supp}\,}

\def\1{\mathbf 1}

\def\Z{{\mathbb Z}}

\def\J{{\mathcal J}}

\def\np{\mathcal{NP}}

\def\be{\begin{equation}}
\def\ee{\end{equation}}

\def\J{{\mathcal J}}

\def\poly{{\mathcal P}}
\def\Q{{\mathcal P}}

\newtheorem{thm}{Theorem}[section]

\theoremstyle{definition}

\theoremstyle{remark}

\newtheorem{preremark}[thm]{Remark}
\newtheorem{preex}[thm]{Example}

\newenvironment{remark}{\begin{preremark}}{\qed\end{preremark}}

\numberwithin{equation}{section}

\begin{document}

\title[Some variants of Macaulay's and Max Noether's Theorems]{Some variants of Macaulay's and \\ Max Noether's Theorems}

\date{\today}
\thanks{The author was partially supported by the Swedish Research Council and NSF grant DMS-0901073} 

\author{Elizabeth Wulcan}

\address{Dept of Mathematics, University of Michigan, Ann Arbor \\ MI 48109-1043\\ USA}

\email{wulcan@umich.edu}




\maketitle

\begin{abstract}
We use residue currents on toric varieties to obtain bounds on the support of solutions to polynomial ideal membership problems. Our bounds depend on the Newton polytopes of the polynomial systems and are therefore well adjusted to sparse systems of polynomials. We present variants of classical results due to Macaulay and Max Noether. 
\end{abstract}

\smallskip
\begin{center}
\emph{Dedicated to Ralf Fr\"oberg on the occasion of his 65th birthday}
\end{center}

\section{Introduction}\label{intro}
Let $F_1, \ldots, F_m$, and $\Phi$ be polynomials in $\C^n$. Assume that $\Phi$ vanishes on the common zero set of the $F_j$. Then \emph{Hilbert's Nullstellensatz} asserts that there are polynomials $G_1,\ldots, G_m$ such that 
\begin{equation}\label{hilbert}
\sum_{j=1}^m F_j G_j = \Phi ^\nu
\end{equation}
for some integer $\nu$ large enough. The following bound of the degrees of the $F_j$ and $\nu$ was obtained by Koll{\'a}r, ~\cite{Kollar}, for $d\neq 2$, and by Jelonek, ~\cite{Jelonek}, for $d=2$ and $m\leq n$: \\
\noindent
\emph{Assume that $\deg F_j \leq d$. Then one can find $G_j$ so that \eqref{hilbert} holds for some $\nu \leq d^{\min(m,n)}$ and
\begin{equation}\label{kollupp}
\deg (F_j G_j) \leq (1 + deg \Phi )d^{\min (m,n)}.
\end{equation}
}
\noindent
For $d=2$ and $m\geq n+1$ the best bound is due to Sombra, \cite{Sombra}: the factor $d^{\min (m,n)}$ in \eqref{kollupp} should then be replaced by $2^{n+1}$. Koll{\'a}r's and Jelonek's bounds are sharp; the original formulations also take into account different degrees of the $F_j$. 
In many cases, however, one can do much better. 
Classical results due to Max Noether, ~\cite{Not}, and Macaulay, ~\cite{Mac}, show that the bounds can be substantially improved if (the homogenizations of) the $F_j$ have no zeros at infinity. The aim of this note is to use multidimensional residues on toric varieties to obtain some variants of these 
results.

Multidimensional residues have been used as a tool to solve polynomial ideal membership problems by several authors, see for example \cite{BGVY}. 
In ~\cite{A3} Andersson used residue currents on manifolds to obtain effective solutions; in particular, Macaulay's and Max Noether's results follow by applying his methods to complex projective space. 

Recall that the \emph{support} $\supp F$ of a polynomial $F=\sum_{\alpha\in\Z^n} c_\alpha z^\alpha =\sum_{\alpha\in\Z^n} c_\alpha z_1^{\alpha_1}\cdots z_n^{\alpha_n}$ in $\C^n$ is defined as $\supp F=\{\alpha\in\Z^n \text{ such that } c_\alpha \neq 0\}$ and that the \emph{Newton polytope} $\np(F_1,\ldots, F_m)$ of polynomials $F_1,\ldots, F_m$ is the convex hull of $\bigcup_j \supp F_j$ in $\R^n$. 
In particular, a polynomial of degree $d$ has support in $d\Sigma^n$, where $\Sigma^n$ is the $n$-dimensional simplex in $\R^n$ with the origin and the unit lattice points $e_1=(1,0,\ldots,0), e_2=(0,1,0,\ldots,0), \ldots, e_n=(0,\ldots, 0, 1)$ as vertices. 

Using techniques from toric geometry Sombra ~\cite{Sombra} obtained a sparse effective Nullstellensatz, which improves Koll{\'a}r's result when the system of polynomials is sparse, meaning that $\np(F_1,\ldots, F_m)$ is small compared to $d\Sigma^n$. In ~\cite{W} the author used the residue current techniques developed in ~\cite{A3} applied to toric varieties in order to obtain certain sparse effective versions of polynomial ideal membership problems. This note, in which we focus on the case when $F_j$ have no common zeros at infinity, can be seen as an addendum to ~\cite{W}. We will specify in Section ~\ref{results} how no common zeros at infinity should be interpreted. 

We work on toric varieties associated with the Newton polytopes or the support of the ~$F_j$. 
Given a lattice polytope $\poly$, i.e., a polytope in $\R^n$ with vertices in $\mathbb Z^n$, one can construct a toric variety $X_\poly$ and a line bundle $\O(D_\poly)$ on $X_\poly$ whose global sections correspond to polynomials with support in $\poly$, see Section ~\ref{toric}. The toric variety $X_\poly$ is smooth if for each vertex $v$ of $\poly$ the smallest integer normal directions of the facets of $\poly$ containing $v$ form a base for $\Z^n$, see ~\cite[p.~29]{F}. We then say that the lattice polytope $\poly$ is \emph{smooth} (or \emph{Delzant}) with respect to the lattice $\Z^n$.

The following sparse version of Macaulay's Theorem is due to Castryck-Denef-Vercauteren, ~\cite{CDV}. 
\begin{thm}\label{mac}
Let $F_1, \ldots, F_m$, and $\Phi$ be polynomials in $\C^n$. Assume that the $F_j$ have no common zeros even at infinity, and that $\supp\Phi\subseteq e\np(F_1,\ldots, F_m)$, where $e\np(F_1,\ldots, F_m)$ is a lattice polytope. 
Then there are polynomials $G_j$ that satisfy 
\begin{equation}\label{ideal}
\sum_{j=1}^mF_jG_j=\Phi
\end{equation}
and
\begin{equation*}
\supp (F_j G_j)\subseteq \max(n+1,e) \np(F_1,\ldots, F_m).
\end{equation*}
\end{thm}

In particular, one can find polynomials $G_j$ that satisfy 
\begin{equation}\label{ett}
\sum_{j=1}^mF_jG_j=1
\end{equation}
and
\begin{equation}\label{lasse}
\supp (F_j G_j)\subseteq (n+1) \np(F_1,\ldots, F_m).
\end{equation}
Macaulay's Theorem, ~\cite{Mac}, corresponds to the case when $\Q=d\Sigma^n$, i.e., $\deg F_j\leq d$. Then \eqref{lasse} reads $\deg(F_jG_j)\leq (n+1)d$, which is slightly worse that Macaulay's original result: \\
\noindent
\emph{
Assume that $F_j$ have no common zeros even at infinity (in $\P^n$). Then one can find $G_j$ that satisfy \eqref{ett} and $\deg (F_j G_j) \leq (n+1)d-n$.}\\
\noindent

Theorem ~\ref{mac} can be seen as a special case of the following sparse version of Max Noether's Theorem, ~\cite{Not}. Let $(F)$ denote the ideal generated by $F_1,\ldots,F_m$.
\begin{thm}\label{maxthm}
Let $F_1, \ldots, F_m$ be polynomials in $\C^n$ and let $\Q$ be a smooth lattice polytope that contains the origin and the support of the $F_j$ and the coordinate functions $z_1,\ldots, z_n$. Assume that the $F_j$ have no common zeros at infinity. Then there is a number $\nu_F$, such that if $\Phi\in (F)$ satisfies that $\supp \Phi\subseteq e\Q$, where $e\Q$ is a lattice polytope, then there are polynomials $G_j$ that satisfy \eqref{ideal}
and
\begin{equation}\label{fetflasket}
\supp (F_j G_j)\subseteq \max (\nu_F, e) \Q. 
\end{equation}
\end{thm}
In fact, Theorem ~\ref{maxthm} is a sparse version of a result in the forthcoming paper ~\cite{AW3}. As Theorem ~\ref{maxthm} is stated above the common zero set of the $F_j$ has to be discrete. It is, however, possible to replace the assumption that the $F_j$ lack common zeros at infinity by a less restrictive assumption, see Remark ~\ref{relaxa}.

The reason that we require $\Q$ to be smooth in Theorem ~\ref{maxthm} is that we need a certain line bundle to be ample, see Section \ref{results}. For example, $\Q=d\Sigma^n$ is smooth; with this choice \eqref{fetflasket} reads $\deg (F_jG_j)\leq \max (\nu_F d, \deg\Phi)$. 

Theorem ~\ref{maxthm} is a variant of Max Noether's Theorem, \cite{Not}, in the sense that $\Phi$ is assumed to be in $(F)$ and the $F_j$ are assumed to have no zeros at infinity. In the original formulation, $F_1,\ldots, F_m$ are moreover assumed to form a \emph{complete intersection}, i.e., the codimension of $\{F_1=\ldots =F_m=0\}$ is $m$: \\
\noindent
\emph{
Assume that the zero-set of $F_1,\ldots, F_n$ is discrete and contained in $\C^n$ and that $\Phi\in(F)$. Then there are $G_j$ that satisfy \eqref{ideal} and $\deg (F_j G_j) \leq \deg\Phi$.}\\
\noindent
Note that if $\supp \Phi$ (or $\deg\Phi$) is large enough, then the bound \eqref{fetflasket} coincides with Max Noether's bound; indeed $\nu_F$ only depends on the $F_j$. In ~\cite[Theorem 1.2]{W} was presented a sparse versions of Noether's Theorem, which essentially says, that if the $F_j$ form a complete intersection, then Theorem ~\ref{maxthm} holds with $\nu_F=0$. To be precise, the polytope $e\Q$ has to satisfy an additional condition.

If the $F_j$ lack common zeros, then Theorem ~\ref{mac} says that we can choose $\nu_F=n+1$. In general, we do not have an explicit description of $\nu_F$, see the discussion after the proof of Theorem ~\ref{maxthm}.

\smallskip 

Recall that the polynomial $\Phi$ lies in the \emph{integral closure} of $(F)$ if $\Phi$ satisfies a monic equation $\Phi^r+H_1 \Phi^{r-1} + \cdots + H_r=0$, where $H_j\in(F)^j$ for $1\leq j\leq r$ or, equivalently, if $\Phi$ locally satisfies $|\Phi|\leq C|F|$, where $|F|^2=|F_1|^2+\cdots + |F_m|^2$. If $\Phi$ is in the integral closure of $(F)$, then the Brian{\c c}on-Skoda Theorem,~\cite{BS}, asserts that one can solve \eqref{hilbert} with $\nu=\min(m,n)$. 
Our next result is a sparse effective Brian{\c c}on-Skoda Theorem, which also can be seen as a generalization of Macaulay's Theorem. Indeed, when the $F_j$ have no common zeros, the assumption below that $\Q$ contains the origin is automatically satisfied and then any polynomial $\Phi$ is in the integral closure of $(F)$. 
\begin{thm}\label{macbs}
Let $F_1, \ldots, F_m$, and $\Phi$ be polynomials in $\C^n$ and let $\Q$ be a lattice polytope that contains the origin and the support of the $F_j$. Assume that the $F_j$ have no common zeros at infinity. Moreover assume that $\Phi$ is in the integral closure of $(F)$ and that $\supp \Phi\subseteq e\Q$, where $e\Q$ is a lattice polytope. Then there are polynomials $G_j$ that satisfy 
\begin{equation}\label{bs}
\sum_{j=1}^m F_j G_j = \Phi^n
\end{equation} 
and
\begin{equation}\label{basse}
\supp (F_j G_j)\subseteq \max (n+1, ne) \Q.
\end{equation}
\end{thm}
The assumption that the $F_j$ have no common zeros at infinity could be replaced by a less restrictive assumption, see Remark ~\ref{relaxa}. 
If $\Q=d\Sigma^n$, then \eqref{basse} reads $\deg(F_j G_j)\leq \max ((n+1)d, n \deg \Phi)$.

\smallskip

Morally, Theorems ~\ref{maxthm} and ~\ref{macbs} say that when the $F_j$ have no zeros at infinity and $\supp \Phi$ is large enough compared to $\supp F_j$, then the bounds on $\supp (F_jG_j)$ in \eqref{ideal} and \eqref{bs} are as good as possible; in fact, $\supp (F_jG_j)$ is then bounded by $\supp \Phi$ and $\supp \Phi^n$, respectively. Andersson-G\"otmark, ~\cite[Thm~1.3]{AG}, and Hickel ~\cite[Thm~1.1]{H} proved effective Max Noether's and Brian{\c c}on-Skoda Theorem's, respectively, in which they allow common zeros at infinity. Then typically terms of size $d^n$ appear, cf.  \eqref{kollupp}.

\smallskip

Let us sketch the idea of the proofs of our results. A standard way of reformulating the kind of division problems we consider is the following. There are polynomials $G_j$ that satisfy \eqref{hilbert} and $\supp (F_j G_j)\subseteq c \poly$ if and only if there are sections $g_j$ of line bundles $\O(D_{(c-1)\poly})$ over $X_\poly $ such that 
\begin{equation}\label{mja}
\sum_{j=1}^m f_j g_j = \psi,
\end{equation}
where $f_j$ and $\psi$ are sections of line bundles $\O(D_\poly)$ and $\O(D_{c\poly})$ over $X_\poly$ corresponding to $F_j$ and $\Phi^\nu$, respectively. Now there is a local solution to \eqref{mja} on $X_\poly$ if $\psi$ annihilates a certain residue current, see Section ~\ref{rescurr}. 
To obtain a global solution to \eqref{mja} the constant $c$ has to be large enough so that certain Dolbeault cohomology on $X_\poly$ vanishes. By analyzing when these conditions are satisfied we obtain our results.

The proofs of Theorems ~\ref{mac}-~\ref{macbs} occupy Section ~\ref{results}. In sections ~\ref{rescurr} and ~\ref{toric} we provide some necessary background on residue currents and toric varieties, respectively.

\section{Residue currents}\label{rescurr}
Let $f_1, \ldots, f_m$ be holomorphic functions whose common zero set $V_f=\{f_1=\ldots=f_m=0\}$ has codimension $m$. Then the \emph{Coleff-Herrera product}, introduced in ~\cite{CH},
\begin{equation*}
R^f_{CH} = \dbar \left [\frac{1}{f_1}\right ]\wedge \cdots \wedge \dbar \left [\frac{1}{f_m}\right ],
\end{equation*}
represents the ideal $(f)$ generated by the $f_j$ in the sense that it has support on $V_f$ and moreover a holomorphic function $\psi$ is in $(f)$ if and only if the current $\psi R^f_{CH}$ vanishes, see ~\cite{DS, P}.

When $\codim V_f < m$, there is no such canonical residue current associated with $f_1,\ldots, f_m$. Passare-Tsikh-Yger, ~\cite{PTY}, constructed residue currents by means of the Bochner-Martinelli kernel that generalize the Coleff-Herrera product to when the codimension of $V_f$ is arbitrary. Their construction was later developed by Andersson, ~\cite{A}, and by Andersson and the author, ~\cite{AW}. 

\begin{thm}\label{residysats}
Assume that $E_0, E_1, \ldots, E_N$ are Hermitian holomorphic vector bundles over a complex manifold $X$ of dimension $n$, and assume that $E_0$ has rank 1. Moreover assume that the complex  
\begin{equation}\label{plex}
0\longrightarrow E_N\stackrel{f^N}{\longrightarrow}\ldots 
\stackrel{f^3}{\longrightarrow}E_2
\stackrel{f^2}{\longrightarrow}E_1
\stackrel{f^1}{\longrightarrow}E_0
\end{equation}
is exact outside an analytic set $Z$ of positive codimension. Then one can construct an $\End (\bigoplus_k E_k)$-valued residue current $R$ on $X$, which has support on $Z$ and satisfies the following:
\begin{enumerate}
\item[(a)]
If $\psi$ is a holomorphic section of $E_0$ that annihilates $R$, i.e., the current $R \psi$ vanishes, then $\psi$ is in the ideal sheaf $\Im f^1$ generated by the image of $f^1$.  
\item[(b)]
If the associated complex of locally free sheaves of $\O$-modules of sections of $E_k$ 
\begin{equation}\label{sheaf}
0\longrightarrow\O(E_N)\stackrel{f^N}{\longrightarrow}\ldots\stackrel{f^2}{\longrightarrow}
\O(E_1)\stackrel{f^1}{\longrightarrow}\O(E_0)
\end{equation}
is exact, then $\psi\in\mathcal \Im f^1$ if and only if $R\psi=0$. 
\item[(c)]
Assume that $f$ is a holomorphic section of a Hermitian vector bundle $E$ of rank $m$ over $X$ and that \eqref{plex} is the Koszul complex of $f$, i.e., $E_k=\Lambda^k E^*$ and $f^k$ is contraction (interior multiplication) with $f$. Moreover assume that $\psi$ locally satisfies that 
\begin{equation*}
|\psi|\leq C |f|^{\min (m,n)}
\end{equation*}
for some constant $C$. Then $R\psi=0$.
\end{enumerate}
\end{thm}

The idea of the proof of Theorem \ref{residysats} is that outside $Z$ one can obtain a local holomorphic solution to the division problem 
\begin{equation}\label{delta}
f^1g=\psi
\end{equation}
by means of \eqref{plex};
here $\psi$ is a section of $E_0$ and $g$ a section of $E_1$. The residue current $R\psi$ appears as an obstruction when one tries to extend the solution from $X\setminus Z$ to $X$; we refer to ~\cite{A} and ~\cite{AW} for details. 

The explicitness of the current $R$ of course directly depends on the explicitness of \eqref{plex}. If \eqref{plex} is the Koszul complex of $f$, then $R$ has support on the zero locus $V_f$ of $f$ and locally the coefficients of $R$ are the residue currents introduced by Passare-Tsikh-Yger, ~\cite{PTY}. In particular, if $\codim V_f=m$, then $R$ is locally a Coleff-Herrera product. Note that in this case $\Im f^1$ is the ideal sheaf $\mathcal J (f)$ generated by $f$. 

Morally, the residue current $R$ is the obstruction to solve \eqref{delta} locally. To obtain a global solution one also needs certain $\dbar$-cohomology on $X$ to vanish. 
The construction of the currents in ~\cite{AW} implies the following, cf.  ~\cite[Prop.~6.1]{AW}:

\begin{thm}\label{forsvinnande}
Let $L$ be a line bundle over $X$. Assume that 
\begin{equation}\label{hos}
H^{0,q}(X, L\otimes E_{q+1})=0
\end{equation}
for $1\leq q \leq \min (N-1, n)$. Let $\psi$ be a holomorphic section of $L\otimes E_0$. 
If $R\psi=0$, then there is a global section $g$ of $L\otimes E_1$ that satisfies \eqref{delta}.
\end{thm}

The current $R$ allows for multiplication with characteristic functions of varieties and more generally constructible sets in such a way that ordinary calculus rules hold, see ~\cite{AW2}. In particular, if $V\subseteq X$ is a variety, then $R \psi=0$ if and only if $(\1_V R)\psi=0$ and $(\1_{X\setminus V}R)\psi=0$. Moreover $R$ is said to have the \emph{Standard Extension Property (SEP)} in the sense of Bj\"ork, ~\cite{Bj}, if $\1_W R=0$ for all subvarieties $W \subset V_f$ of positive codimension.

\section{Toric varieties from polytopes}\label{toric}
For a general reference on toric varieties, see ~\cite{F}. A toric variety can be constructed from a \emph{fan} $\Delta$, which is a certain collection of $\mathbb Z^n$ cones, by gluing together copies of $\C^n$ corresponding to the $n$-dimensional cones of $\Delta$; we denote the resulting toric variety by $X_\Delta$. 
Let $\poly$ be a lattice polytope in $\R^n$. Then $\poly$ determines a fan $\Delta_\poly$, the so-called \emph{normal fan} of $\poly$, whose rays correspond to the normal directions of the faces of maximal dimension of $\poly$. The corresponding toric variety $X_\poly=X_{\Delta_\poly}$ is projective, see \cite[Section~VII.3]{Ewald}.

A toric variety $X_\Delta$ is smooth if and only if each cone in $\Delta$ is generated by a part of a basis for the lattice $\Z^n$. Such a fan is said to be \emph{regular}. The fan $\Delta_\poly$ is regular precisely when $\poly$ is smooth, cf. the introduction. For each fan $\Delta$ there exists a refinement $\widetilde \Delta$ of $\Delta$ such that $X_{\widetilde \Delta} \to X_\Delta$ is a resolution of singularities. Also if $\Delta_1$ and $\Delta_2$ are two different fans, there exists a regular fan $\widetilde\Delta$ that refines both $\Delta_1$ and $\Delta_2$. If $\Delta$ is a refinement of $\Delta_\poly$ we say that $\Delta$ and $\poly$ are  \emph{compatible}.

Assume that $\poly$ is compatible with $\Delta$. Then $\poly$ defines a divisor $D_\poly$ on $X_\Delta$ such that the global holomorphic sections of the line bundle $\O(D_\poly)$ correspond precisely to the polynomials with support in $\poly$. Moreover $\O(D_\Q)$ is generated by its sections, and if $\Delta=\Delta_\poly$, then $\O(D_\Q)$ is ample. Also, $\O(D_\Q)\otimes\O(D_{\mathcal Q})=\O(D_{\Q+\mathcal Q})$. 

If $\Delta$ is compatible with a polytope and $L$ is a line bundle over $X_\Delta$ that is generated by its sections, then $H^{0,q}(X_\Delta,L)=0$ for all $q\geq 1$.

In the situation of Theorems ~\ref{maxthm} and ~\ref{macbs} we want to consider toric varieties that are compactifications of $\C^n$. 
Assume that $\Delta$ contains the first orthant $\sigma_0$ as an $n$-dimensional cone; observe that if $\Q\subseteq \R^n_+$ contains the origin, then one can find such a $\Delta$, which is regular and compatible with $\Q$. Then we can identify the corresponding affine chart $\U_{\sigma_0}$ with $\C^n$; we refer to the complement $X_\Delta\setminus\U_{\sigma_0}$ as the \emph{variety at infinity} and denote it by $V_\infty$. If $\poly$ is compatible with $\Delta$ and moreover contains the origin, then in local coordinates in $\U_{\sigma_0}=\C^n$, a section $\psi$ of $\O(D_\poly)$ coincides with the corresponding polynomial $\Psi$ in $\C^n$, so that $\psi$ can really be seen as a homogenization of $\Psi$, see ~\cite{Cox} and also ~\cite[Section~3.4]{W}.

\section{Proofs}\label{results}
In Theorem ~\ref{mac} the $F_j$ are assumed to have no common zeros even at infinity. This should be interpreted as that the corresponding sections $f_j$ of $\O(D_{\Q})$ lack common zeros in $X_\Delta$, where $\Delta$ is compatible with $\Q=\np(F_1,\ldots, F_m)$. Observe that whether the $f_j$ have common zeros in $X_\Delta$ in fact only depends on $\Q$ and not on the particular choice of $\Delta$, as long as it is compatible with $\Q$. In Theorems ~\ref{maxthm} and ~\ref{macbs}, $\Q$ is assumed to contain the origin. It follows that $\Delta$ can be chosen compatible with $\Q$ so that it contains the first orthant as a cone. The assumption that the $f_j$ lack common zeros at infinity should be interpreted as that, given such a $\Delta$, the corresponding sections of $\O(D_\Q)$ lack common zeros at $V_\infty$ in $X_\Delta$.

Consider polynomials $F_j$ with support in polytopes $\Q_j$. 
Whether or not the $F_j$, or rather the corresponding sections $f_j$  of line bundles $\O(D_{\Q_j})$, have common zeros (at infinity) clearly depends on the polytopes $\Q_j$. Assume that $f_j$ are sections of a line bundle $\O(D_\Q)$ over $X_\Delta$, where $\Delta$ is compatible with $\Q$. Then the $f_j$ do have common zeros unless $\Q=\np(F_1,\ldots, F_m)$ and they have common zeros at infinity unless $\Q$ is the convex hull of the Newton polytope and the origin. On the other hand, any generic choice of $n+1$ sections of $\O(D_\Q)$ will lack common zeros and any choice of $n$ polynomials with support in $\Q$ will lack common zeros at $V_\infty$, see for example ~\cite[Section 6.2]{W} or ~\cite[Lma~4.1]{Tuitman}. Thus the sparse versions of Macaulay's and Max Noether's results generalize their classical counterparts in the sense that they apply to more general situations.

Theorem ~\ref{mac} is a consequence of the following more general result, which is due to Tuitman ~\cite{Tuitman}; we include a proof for completeness. Recall that the polytope $\mathcal Q$ is a \emph{summand} of the polytope $\Q$ if there exist another a polytope $\mathcal S$ such that $\Q=\mathcal Q+\mathcal S$.

\begin{thm}\label{tuit}[Tuitman ~\cite{Tuitman}]
Let $F_1,\ldots, F_m$, and $\Phi$ be polynomials in $\C^n$. Let $\Q_j$ and $\Q$ be polytopes that contain the support of the $F_j$ and $\Phi$, respectively. Assume that the $F_j$ have no common zeros even at infinity, meaning that the corresponding sections of line bundles $\O(D_{\Q_j})$ over a toric variety lack common zeros. Assume that $\Q_{j_1}+\cdots + \Q_{j_{q}}$ is a summand of $\Q$ for all $1\leq q\leq \min(m,n+1)$ and $\J=\{j_1,\ldots, j_q\}\subseteq \{1,\ldots, m\}$.
Then there are polynomials $G_j$ that satisfy \eqref{ideal} and
\begin{equation}\label{tuttut}
\supp (F_j G_j)\subseteq \Q.
\end{equation}
\end{thm}
In particular, we can let $\Q=\sum_{j=1}^m\Q_j$. 
Also, if we choose $\Q$ as $\max(n+1,e)\np(F_1,\ldots,F_m)$ we get back Theorem ~\ref{mac}.

\begin{proof}
Let $\Delta$ be a regular fan that is compatible with $\Q_1,\ldots, \Q_m$, and $\Q$, let $E$ be the bundle $\O(D_{\Q_1})\oplus\cdots\oplus\O(D_{\Q_m})$ over $X_\Delta$, and let $L$ be the line bundle $\O(D_{\Q})$. We identify polynomials with support in $\Q_j$ and $\Q$ with sections of $\O(D_{\Q_j})$ and $L$, respectively. Accordingly, let $f_j$, $f$, and $\psi$ be the sections of $\O(D_{\Q_j})$, $E$, and $L$ corresponding to $F_j$, the tuple $F_1,\ldots, F_m$, and $\Phi$, respectively.

Let \eqref{plex} be the Koszul complex of $f$ and let $R$ be the associated residue current. By assumption, the $f_j$ have no common zeros, and hence $R=0$. 

Now 
\begin{equation}\label{msri}
L\otimes E_q=L\otimes \Lambda^qE^*=\bigoplus_{|\J|=q}\O(D_\Q-(D_{\Q_{j_1}}+\cdots + D_{\Q_{j_q}})),
\end{equation}
Since for each term in the right hand side of \eqref{msri}, $\Q_{j_1}+\cdots + \Q_{j_q}$ is a summand of $\Q$, $\O(D_{\Q-(\Q_{j_1}+\cdots + \Q_{j_q})})$ is generated by its sections, see Section ~\ref{toric}. Hence \eqref{hos} holds for $1\leq q\leq n$, cf. (the proof of) Theorem ~4.1 in \cite{W}. 

Now Theorem ~\ref{forsvinnande} asserts that we can find a section $g=(g_1,\ldots, g_m)$ of $L\otimes E^*$ that satisfies \eqref{delta}, and thus polynomials $G_j$ that satisfy \eqref{ideal} and \eqref{tuttut}. 
\end{proof} 

The original proof by Tuitman is very similar to our proof. In fact, the residue current does not really play a role in our proof, since it trivially vanishes.

Theorem ~\ref{macbs} is proved along the same lines as Theorem ~\ref{mac}, using  residue currents constructed from the Koszul complex. It would be possible to give a more general formulation of Theorem ~\ref{macbs}, that would take into account that the $F_j$ might have different supports, as was done in Theorem ~\ref{tuit}.
\begin{proof}[Proof of Theorem ~\ref{macbs}]
Let $\Delta$ be a regular fan that is compatible with $\Q$ and that contains the first orthant as cone. Moreover, let $E$ be the vector bundle $\O(D_\Q)^{\oplus m}$ over $X_\Delta$, and let $L$ be the line bundle $\O(D_{\max(n+1,ne)\Q})$. Let $f_j$, $f$, and $\psi$ be the sections of $\O(D_\Q)$, $E$, and $L$ corresponding to $F_j$, the tuple $F_1,\ldots, F_m$, and $\Phi^n$, respectively.

Let \eqref{plex} be the Koszul complex of $f$ and let $R$ be the associated residue current. By assumption, the $f_j$ have no common zeros at infinity, and hence $\1_{V_\infty}R=0$. Moreover, since $\Phi$ is in the integral closure of $(F)$ in $\C^n$, $(\1_{\C^n}R)\psi=0$ by Theorem ~\ref{residysats} ~(c) and the end of Section ~\ref{toric}. 

By Section ~\ref{toric}, $L\otimes E_q=L\otimes \Lambda^qE^*$ is a direct sum of line bundles $\O(D_{(\max(n+1,ne)-q)\Q})$, and since $\O(D_{c\Q})$ is generated by its sections if $c\geq 0$, by Section ~\ref{toric}, \eqref{hos} holds for $1\leq q\leq n$. 

Now Theorem ~\ref{forsvinnande} asserts that we can find a section $g=(g_1,\ldots, g_m)$ of $L\otimes E^*$ that satisfies \eqref{delta}, and thus polynomials $G_1,\ldots, G_m$ in $\C^n$ that satisfy \eqref{bs} and \eqref{basse}.
\end{proof}

\begin{proof}[Proof of Theorem ~\ref{maxthm}]
Let $E$ be the vector bundle $\O(D_\Q)^{\oplus m}$ over $X_\Q$, and let $f$ be the section of $E$ corresponding to $F_1,\ldots,F_m$. Let $E_0$ be the trivial bundle of rank 1 over $X_\Q$, let $E_1=E^*$, and let $f^1$ be multiplication with $f$. Since $X_\Q$ is projective, $E_1\stackrel{f^1}{\to} E_0$ can be continued to a complex \eqref{plex}, such that the associated complex \eqref{sheaf} is exact, see for example ~\cite[Ex.~1.2.21]{Laz1}. Since, by assumption, $\Q$ is smooth, the line bundle $\O(D_\Q)$ over $X_\Q$ is ample and thus for some large enough number $\nu_F$, $H^{0,q}(X_\Q,\O(D_\Q)^{\otimes \nu} \otimes E_{q+1})=0$ for $1\leq q\leq \min (N-1,n)$ and $\nu\geq\nu_F$. In particular, $L=\O(D_{\max(\nu_F,e)\Q})$ satisfies \eqref{hos} for $1\leq q\leq \min (N-1,n)$. 

The assumption that $\Q$ contains the origin and the support of the coordinate functions $z_1,\ldots, z_n$ implies that the first orthant in $\R^n$ is a cone of $\Delta_\Q$. Let $R$ be the residue current associated with \eqref{plex} and let $\psi$ be the section of $L$ corresponding to $\Phi$. 
By assumption, the $f_j$ have no common zeros at infinity, and hence $\1_{V_\infty}R=0$. Moreover, since \eqref{sheaf} is exact and $\Phi\in(F)$ in $\C^n$, $(\1_{\C^n}R)\psi=0$ by Theorem ~\ref{residysats} (b) and the end of Section ~\ref{toric}. 

Now Theorem ~\ref{forsvinnande} asserts that we can find a section $g=(g_1,\ldots,g_m)$ of $L\otimes E_1=L\otimes E^*$ that satisfies \eqref{delta}, and thus polynomials $G_1,\ldots, G_m$ in $\C^n$ that satisfy \eqref{ideal} and \eqref{fetflasket}.
\end{proof} 

The constant $\nu_F$ in Theorem ~\ref{maxthm} depends on the degrees of the mappings in the resolution \eqref{plex}, which are closely related to the Castelnuovo-Mumford regularity of $(F)$, see ~\cite[Chapter 20.5]{Eisenbud}.

\begin{remark}\label{relaxa}
Observe that the proofs of Theorems ~\ref{maxthm} and ~\ref{macbs} only use that $R$ vanishes along $V_\infty$, i.e., $\1_{V_\infty}R=0$. In fact, this allows us to replace the assumptions that the $F_j$ lack common zeros at infinity by less restrictive assumptions. 

Let $Z_k$ be the set where the mapping $f^k$ in \eqref{plex} does not have optimal rank. When \eqref{sheaf} is exact $R$ admits a decomposition $R=\sum_k\1_{Z_k\setminus Z_{k-1}}R$, where $\1_{Z_k\setminus Z_{k-1}}R$ has support on and the SEP with respect to $Z_k$, see ~\cite[Ex.~7]{AW2}. Thus in Theorem ~\ref{maxthm} we could replace the assumption that the $F_j$ lack common zeros at infinity by the assumption that the $Z_k$ have no irreducible components contained in $V_\infty$.

Let $\{V_j\}$ be the set of so-called \emph{distinguished subvarieties} of $\mathcal J(f)$, see ~\cite[p.~263]{Laz}, and let $R$ be the residue current constructed from the Koszul complex of $f$. It follows from the construction that $R$ admits a decomposition $R=\sum\1_{V_j}R$, where $\1_{V_j}R$ has support on and the SEP with respect $V_j$, see for example ~\cite{AG}. Hence in Theorem ~\ref{macbs} we could replace the assumption that $F_j$ lack common zeros at infinity by the assumption that $\mathcal J(f)$ has no distinguished subvarieties contained in ~$V_\infty$.
\end{remark}

Thanks to the referee for many helpful suggestions.

\def\listing#1#2#3{{\sc #1}:\ {\it #2},\ #3.}

\end{document}